\documentclass{amsart}
\usepackage{amsmath, amscd, amssymb, amsthm}
\usepackage{bbm}
\usepackage{latexsym}
\usepackage{amsfonts}
\usepackage{graphicx}
\usepackage[all,cmtip]{xy}

\usepackage[%dvipdfm,
colorlinks=true, linkcolor=black,breaklinks=true,
urlcolor=blue, citecolor=black]{hyperref}%
\usepackage{geometry}
\newtheorem{theorem}{Theorem}%[section]
\newtheorem{lemma}{Lemma}

\newtheorem{proposition}[theorem]{Proposition}

\newtheorem{conjecture}{Conjecture}

\newcommand{\be}{\begin{equation}}
\newcommand{\ee}{\end{equation}}
\newcommand{\bea}{\begin{eqnarray}}
\newcommand{\eea}{\end{eqnarray}}

\def\XXint#1#2#3{{\setbox0=\hbox{$#1{#2#3}{\int}$ }
\vcenter{\hbox{$#2#3$ }}\kern-.6\wd0}}

\begin{document}

\title[Almost abelian groups with constant $H$]{A note on almost abelian groups with constant holomorphic sectional curvature}

\author{Yulu Li}
\address{Yulu Li. School of Mathematical Sciences, Chongqing Normal University, Chongqing 401331, China}
\email{{1320779072@qq.com}}

\author{Fangyang Zheng} \thanks{Li is supported by Chongqing Graduate Student Research Innovation Program (Grant No. CYB23230). The corresponding author Zheng is supported by NSFC grants \# 12141101 and \# 12471039, by Chongqing Normal University grant 24XLB026, and is supported by the 111 Project D21024.}
\address{Fangyang Zheng. School of Mathematical Sciences, Chongqing Normal University, Chongqing 401331, China}
\email{{20190045@cqnu.edu.cn}}


\begin{abstract}
A long-standing conjecture in non-K\"ahler geometry states that if the Chern (or Levi-Civita) holomorphic sectional curvature of a compact Hermitian manifold is a constant $c$, then the metric must be K\"ahler when $c\neq 0$ and must be Chern (or Levi-Civita) flat when $c=0$. The conjecture is known to be true in dimension 2 by the work of Balas-Gauduchon, Sato-Sekigawa, and Apostolov-Davidov-Muskarov in the 1980s and 1990s. In dimension 3 or higher, the conjecture is still open except in some special cases, such as for all twistor spaces by Davidov-Grantcharov-Muskarov, for locally conformally K\"ahler manifolds (when $c\leq 0$) by Chen-Chen-Nie, etc. In this short note, we consider compact quotients $G/\Gamma$ where $G$ is a Lie group equipped with a left-invariant complex structure and a compatible left-invariant metric, and $\Gamma$ is a discrete subgroup. We confirm the conjecture when the Lie algebra ${\mathfrak g}$ of $G$  either is almost abelian, or contains a $J$-invariant abelian ideal of codimension 2. 
\end{abstract}

\maketitle

%\tableofcontents

\markleft{Li and Zheng}
\markright{Almost abelian groups with constant holomorphic sectional curvature}

\section{Introduction and statement of result}

In differential geometry, the simplest kind of manifolds are those with constant curvature. In the Riemannian setting, complete Riemannian manifolds with constant sectional curvature are called {\em space forms}. Their universal covers are the sphere $S^n$, the Euclidean space ${\mathbb R}^n$, or the hyperbolic space ${\mathbb H}^n$, equipped with (a constant multiple of) the standard metrics. For K\"ahler manifolds, the sectional curvature could no longer be constant unless it is flat, so instead one requires the holomorphic sectional curvature to be constant. Analogous to the real case, complete K\"ahler manifolds with constant  holomorphic sectional curvature are called {\em complex space forms}, and their universal covers are the complex projective space ${\mathbb C}{\mathbb P}^n$, the complex Euclidean space ${\mathbb C}^n$, or the complex hyperbolic space ${\mathbb C}{\mathbb H}^n$, equipped with (a constant multiple of) the standard metrics. 

In the complex case, it is a natural question to ask what would happen if one drops the K\"ahlerness assumption, namely, what kind of Hermitian manifolds would have constant holomorphic sectional curvature? Of course, one needs to specify which connection is used for the curvature. Given a Hermitian manifold $(M^n,g)$, denote by $\nabla$ the Chern connection and $R$ its curvature. Then the (Chern) holomorphic sectional curvature is defined by
$ H(X) =R_{X\overline{X}X\overline{X}}/|X|^4$, 
where $X\neq 0$ is any complex tangent vector of type $(1,0)$. Similarly, if we denote by $\nabla^r$ the Levi-Civita (also known as Riemannian) connection, and $R^r$ its curvature tensor, then the (Levi-Civita)  holomorphic sectional curvature $H^r$ is defined similarly as above when $R$ is replaced by $R^r$. Note that when the metric $g$ is K\"ahler, it is well-known that $\nabla =\nabla^r$, hence $R=R^r$. Conversely, when $R=R^r$, it was proved by Yang and the second named author in \cite[Theorem 1]{YZ} that $g$ must be K\"ahler (see Proposition \ref{prop0} in \S 2 for a related observation). A long-standing open question in non-K\"ahler geometry is the following:

\begin{conjecture}[{\bf Constant Holomorphic Sectional Conjecture}] \label{conj1}
Let $(M^n,g)$ be a compact Hermitian manifold with  $H$  (or $H^r$) equal to a constant $c$. If $c\neq 0$, then $g$ must be K\"ahler (hence $(M^n,g)$ is a complex space form). If $c=0$, then $g$ must be Chern (or Levi-Civita) flat, namely, $R=0$ (or $R^r=0$).
\end{conjecture}

Note that the compactness assumption in the conjecture is necessary, otherwise there are counterexamples. By the classic result of Boothby \cite{Boothby}, compact Chern flat manifolds are exactly compact quotients of complex Lie groups (equipped with left-invariant metrics). Also, by the classic theorem of Bieberbach, compact Riemannian flat manifolds are finite undercovers of flat tori. Given a flat torus $(T^{2n},g)$ of real dimension $2n$, there are many complex structures $J$ compatible with the given flat metric $g$, turning $(T^{2n},J,g)$ into  Hermitian manifolds of complex dimension $n$  that are Levi-Civita (namely, Riemannian) flat. It is not hard to see that $g$ is K\"ahler with respect to $J$ if and only if $(T^{2n},J)$ is a complex $n$-torus, and when $n\leq 2$, these are the only possibilities. For each $n\geq 3$, however, there are always other complex structures $J$ for some flat $2n$-tori to form non-K\"ahler examples of Riemannian flat Hermitian $n$-manifolds. For $n=3$ there is a full classification of all such manifolds by Khan, Yang and the second named author \cite{KYZ}, but for $n\geq 4$, the classification is still unknown.   

Back to Conjecture \ref{conj1}, the $n=2$ case is known to be true. For the Chern connection, the $c\leq 0$ case was proved by Balas and Gauduchon in 1985 in \cite{Balas,BG}, while the $c>0$ case was solved by Apostolov, Davidov and Muskarov in \cite{ADM} in 1996. For the Levi-Civita connection, the $c\leq 0$ case was proved by Sato and Sekigawa in 1990 \cite{SS}, and the $c>0$ case was again solved by \cite{ADM}. 

For $n\geq 3$, Conjecture \ref{conj1} is still largely open except in a number special cases. The first substantial result towards the conjecture is the one obtained by Davidov-Grantcharov-Muskarov \cite{DGM}, in which they showed among other things that the only twistor space with constant holomorphic sectional curvature is the complex space form ${\mathbb C}{\mathbb P}^3$. More recently, Chen-Chen-Nie \cite{CCN} considered locally conformally K\"ahler manifolds and confirmed the conjecture in the $c\leq 0$ cases. They also pointed out the necessity of the compactness assumption in the conjecture by explicit examples. Tang in \cite{Tang} proved the conjecture under the additional assumption that the metric is Chern K\"ahler-like, meaning that the curvature tensor $R$ obeys all the K\"ahler symmetries. 

Zhou and the second named author in \cite{ZhouZ} proved that any compact Hermitian threefold with vanishing real bisectional curvature must be Chern flat. Real bisectional curvature is a curvature notion introduced by X. Yang and the second named author in \cite{YangZ}. It is equivalent to $H$ in strength when the metric is K\"ahler, but is slightly stronger than $H$ when the metric is non-K\"ahler.

In \cite{RZ}, Rao and  the second named author confirmed the conjecture under the additional assumption that $g$ is {\em Bismut-K\"ahler like,} meaning that the curvature tensor $R^b$ of the Bismut connection $\nabla^b$ obeys all K\"ahler symmetries. This is a highly restrictive class of manifolds which has drawn attention in recent years, see for instance \cite{AOUV,  YZZ, ZZJGP, ZZCrelle}. In \cite{CZ1}, Chen and the second named author explored the analogous conjecture for Bismut connection and established the answers for the $2$-dimensional case, utilizing the fundamental work of \cite{ADM}. They explored further in \cite{CZ2} for these questions in the broader set of {\em Bismut torsion parallel} manifolds, meaning those Hermitian manifolds whose Bismut connection has parallel torsion. We refer the readers to \cite{ZhaoZ24} and references therein for more discussion on this interesting class of special Hermitian manifolds. 
 
Note that a major difficulty in dealing with Conjecture \ref{conj1} is due to the algebraic complexity of the curvature tensor of Hermitian manifolds. When the metric is non-K\"ahler, the curvature tensor (for Chern, Levi-Civita, or Bismut connection)  does not obey the K\"ahler symmetries in general, and the holomorphic sectional curvature does not control the entire curvature tensor (instead, it only controls the `symmetrized part' of the curvature tensor). So as a useful testing ground for the conjecture, let us restrict our attentions to the subset of {\em Lie-Hermitian manifolds}, by which we mean compact Hermitian manifolds in the form $M=G/\Gamma$, where $G$ is a Lie group and $\Gamma \subset G$ a discrete subgroup. The complex structure $J$ and metric $g$ on $M$, when lifted onto $G$, are both left-invariant. Note that there are plenty of examples of such $(G,J,g)$ where $J$ is not bi-invariant, namely, $G$ is not a complex Lie group. We will call such a manifold a {\em complex nilmanifold} if $G$ is nilpotent, and a {\em complex solvmanifold} if $G$ is solvable. One can certainly restrict Conjecture \ref{conj1} to the special case of Lie-Hermitian manifolds and ask for the weaker question. Note that the Lie group $G$ is unimodular since it admits a compact quotient. By the result of Hano \cite{Hano}, any unimodular K\"ahler Lie algebra must be flat.  Therefore, the Lie-Hermitian special case of Conjecture \ref{conj1} is reduced to the following

\begin{conjecture} \label{conj2}
Let $(M^n,g)$ be a Lie-Hermitian manifold, namely, a compact Hermitian manifold in the form $M=G/\Gamma$ where $G$ is a Lie group, $\Gamma \subset G$ a discrete subgroup, and the complex structure and metric when lifted to $G$ are left-invariant. If the holomorphic sectional curvature for the Chern (or Levi-Civita) connection is a constant $c$, then $c$ must be zero and $g$ must be Chern (or Levi-Civita) flat.
\end{conjecture}
 
When $g$ is Chern flat, we know that $(M^n,g)$ must be the compact quotient of a complex Lie group $G'$ equipped with a left-invariant metric (compatible with its complex structure). However, our $G$ in Conjecture \ref{conj2} might not be a complex Lie group, and there are indeed examples of such kind. For the Levi-Civita flat case, the result of Milnor \cite[Theorem 1.5]{Milnor} states that, given a Lie group $G$ with a left-invariant metric $g$, it is Levi-Civita flat if and only if the Lie algebra ${\mathfrak g}$ of $G$ is the orthogonal direct sum ${\mathfrak b} \oplus {\mathfrak u}$ where ${\mathfrak b}$ is an abelian subalgebra, ${\mathfrak u}$ is an abelian ideal, and $\mbox{ad}_b$ is skew-adjoint for all $b\in {\mathfrak b}$, namely, 
$$ \langle [b,x],y\rangle + \langle [b,y],x\rangle  = 0, \ \ \ \ \forall \ x,y\in {\mathfrak g}. $$
Note that in this case $G$ is automatically unimodular, that is, $\mbox{tr}(\mbox{ad}_x)=0$ for all $x\in {\mathfrak g}$. Also, $G$ is always $2$-step solvable, but is not nilpotent (unless it is abelian). Lie groups with Hermitian structures that are K\"ahler flat were classified by Barberis-Dotti-Fino in \cite[Prop 2.1, Cor 2,2, Prop 3.1]{BDF}. See also \cite[Appendix]{VYZ} for a summary discussion. 

Regarding Conjecture \ref{conj2}, our previous work \cite{LZ} shows that the Chern connection case holds if $G$ is nilpotent. In fact, we showed that $G$ must be a complex Lie group. It was shown by Chen and the second named author in \cite[Theorem 2 (1)]{CZ3} that the Levi-Civita case holds if $G$ is nilpotent and $J$ is also nilpotent in the sense of \cite{CFGU}. At present we do not know how to prove the Levi-Civita case of Conjecture \ref{conj2} for complex nilmanifolds when the assumption `$J$ is nilpotent'  is dropped. In \cite{HZ}, Huang and the second named author proved the Chern connection case of Conjecture \ref{conj2} for solvable Lie groups $G$ with $J$-invariant commutator.

The main purpose of this article is to prove Conjecture \ref{conj2} in the following two special cases:  when the Lie algebra ${\mathfrak g}$ of $G$ is either almost abelian (meaning that it contains an abelian ideal of codimension $1$), or when it contains an abelian ideal ${\mathfrak a}$ of codimension $2$ satisfying $J{\mathfrak a}={\mathfrak a}$:
\begin{theorem} \label{thm}
Let $(M^n,g)$ be a compact Hermitian manifold with universal cover $(G,J,g)$, where $G$ is a Lie group and both $J$ and $g$ are left-invariant. Assume that the Chern (or Levi-Civita) holomorphic sectional curvature of $g$ is a constant $c$, and  ${\mathfrak a}$ is an abelian ideal in the Lie algebra  ${\mathfrak g}$ of $G$. If either (i) ${\mathfrak a}$ is of codimension $1$, or (ii) ${\mathfrak a}$ is of codimension $2$ such that $J{\mathfrak a}={\mathfrak a}$, then $c=0$ and $g$ is Chern (or Levi-Civita) flat. 
\end{theorem}
In other words, Conjecture \ref{conj2} holds if ${\mathfrak g}$ is either almost abelian or it contains a $J$-invariant abelian ideal of codimension $2$. As we shall see in the proofs, in the Chern connection case, one does not need to assume that ${\mathfrak g}$ is unimodular, namely, $H=c$ implies that $c=0$ and $R=0$. On the other hand, in the Levi-Civita connection case, one does need to assume that ${\mathfrak g}$ is unimodular. In this case $H^r=c$ leads to $c=0$, $R^r=0$, and $g$ is in fact K\"ahler. The conclusion fails when ${\mathfrak g}$ is not unimodular and $c<0$, namely, there are examples of (non-unimodular) Hermitian Lie algebra which is almost abelian and has constant negative Levi-Civita holomorphic sectional curvature, yet the metric is not K\"ahler. We shall see an explicit example of such kind in \S 3.

The article is organized as follows. In the next section, we will set up the notations, collect some known results from existing literature, and prove some preliminary lemmas. In the sections $3$ and $4$ we will give proofs to Theorem \ref{thm} in the almost abelian case and in the case when ${\mathfrak g}$ contains a $J$-invariant abelian ideal of codimension $2$.

\vspace{0.3cm}

\section{Preliminaries }

Let $(M^n,g)$ be a Hermitian manifold. Denote by $\nabla$ the Chern connection and by $T$, $R$ its torsion and curvature tensor,  given respectively by
$$ T(x,y) = \nabla_xy-\nabla_yx-[x,y], \ \ \ \ R(x,y,z,w)=\langle \nabla_x\nabla_yz-\nabla_y\nabla_xz -\nabla_{[x,y]}z, w\rangle ,$$
where $x$, $y$, $z$, $w$ are vector fields on $M^n$ and $g=\langle , \rangle $ is the metric. In terms of type $(1,0)$ tangent vectors, namely, those in the form  $X=x-\sqrt{-1}Jx$ where $x$ is any real tangent vector, the holomorphic sectional curvature $H$ of $R$ would be a constant $c$ if and only if
$$ R_{X\overline{X}X\overline{X}} = c |X|^4  $$
for any type $(1,0)$ tangent vector $X$. Under a local unitary frame $e=\{ e_1, \ldots , e_n\}$, we have
\begin{equation}
H=c \ \Longleftrightarrow \ \widehat{R}_{i\overline{j}k\overline{\ell}}=\frac{c}{2}(\delta_{ij}\delta_{k\ell} + \delta_{i\ell}\delta_{kj}), \label{RhatC}
\end{equation}
where
\begin{equation} \label{Rhat}
\widehat{R}_{i\overline{j}k\overline{\ell}} = \frac{1}{4} \big( R_{i\overline{j}k\overline{\ell}} + R_{k\overline{j}i\overline{\ell}} +  R_{i\overline{\ell}k\overline{j}} +  R_{k\overline{\ell}i\overline{j}} \big) 
\end{equation}
is the symmetrization of $R$. When $g$ is K\"ahler, the well-known K\"ahler symmetry says that $\widehat{R}=R$, and $H$ determines the entire $R$, but for general Hermitian metrics, $H$ can only determine $\widehat{R}$ but not $R$.

Similarly, let us denote by $\nabla^r$ the Levi-Civita connection of $g$ and by $R^r$ the curvature tensor. Denote by $H^r$ its holomorphic sectional curvature, then we have
\begin{equation}
H^r=c \ \Longleftrightarrow \ \widehat{R}^r_{i\overline{j}k\overline{\ell}}=\frac{c}{2}(\delta_{ij}\delta_{k\ell} + \delta_{i\ell}\delta_{kj}). \label{RrhatC}
\end{equation}
Under any unitary frame $e$, write $T(e_i,e_k)=\sum_{j=1}^n T^j_{ik}e_j$, $1\leq i,k\leq n$ for the  Chern torsion components.  The curvature tensors $R$ and $R^r$ are related by the following formula from \cite[Lemma 7]{YZ} (note that in that paper, $R$ and $R^r$ are denoted by $R^h$ and $R$, respectively, while $T^j_{ik}$ is denoted as $2T^j_{ik}$): 

\begin{lemma}[\cite{YZ}] \label{lemma1}
Let $(M^n,g)$ be a Hermitian manifold and $e$ be a local unitary frame. The Chern and Levi-Civita curvature components under $e$ are related by
\begin{eqnarray}
T^{\ell}_{ik,\bar{j}} & = & R_{k\bar{j}i\bar{\ell}} -  R_{i\bar{j}k\bar{\ell}}, \nonumber \\
R^r_{ijk\bar{\ell}} & = &  \frac{1}{2} T^{\ell}_{ij,k} + \frac{1}{4}\big( T^s_{jk}R^{\ell}_{si} -  T^s_{ik}R^{\ell}_{sj} \big) , \nonumber \\
R^r_{i\bar{j}k\bar{\ell}} &= &   R_{i\bar{j}k\bar{\ell}}  + \frac{1}{2}\big( T^{\ell}_{ik,\bar{j}}  + \overline{  T^k_{j\ell , \bar{i}} } \big) + \frac{1}{4}\big( T^s_{ik} \overline{T^s_{j\ell } } -   T^{\ell}_{is} \overline{T^k_{js} } -  T^{j}_{ks} \overline{T^i_{\ell s} }   \big) , \label{eq:4}\\
R^r_{ik\bar{j}\bar{\ell}} & = &  \frac{1}{2}  \big(    T^{\ell}_{ik,\bar{j}} -  T^{j}_{ik,\bar{\ell}}   \big) + \frac{1}{4}\big( 2T^s_{ik} \overline{T^s_{j\ell } } +  T^{j}_{is} \overline{T^k_{\ell s} } +  T^{\ell}_{ks} \overline{T^i_{js} } -   T^{\ell}_{is} \overline{T^k_{js} } -  T^{j}_{ks} \overline{T^i_{\ell s} }   \big) , \nonumber
\end{eqnarray} 
where $s$ is summed up from $1$ to $n$ and indices after comma stand for covariant derivatives with respect to the Chern connection $\nabla$.
\end{lemma}

By (\ref{Rhat}) and (\ref{eq:4}), we immediately get the following:
\begin{lemma} \label{lemma2}
Let $(M^n,g)$ be a Hermitian manifold and $e$ be a local unitary frame. The symmetrization of Levi-Civita curvature has components
\begin{equation*}
\widehat{R}^r_{i\bar{j}k\bar{\ell}} = \widehat{R}_{i\bar{j}k\bar{\ell}} - \frac{1}{8} \sum_{s=1}^n \big( T^{j}_{is} \overline{T^k_{\ell s} } +  T^{\ell}_{ks} \overline{T^i_{js} } +  T^{\ell}_{is} \overline{T^k_{js} } + T^{j}_{ks} \overline{T^i_{\ell s} } \big) , \ \ \ \forall \ 1\leq i,j,k,\ell \leq n. \label{eq:8}
\end{equation*}
\end{lemma}
In particular, we have
\begin{equation}
\left\{ \begin{split}
R^r_{i\bar{i}i\bar{i}} = R_{i\bar{i}i\bar{i}} - \frac{1}{2}\sum_{s=1}^n |T^i_{is}|^2, \ \ \  \ \ \  \ \ \ \ \forall \ 1\leq i\leq n.  \hspace{4cm} \\
\widehat{R}^r_{i\bar{i}k\bar{k}} = \widehat{R}_{i\bar{i}k\bar{k}} - \frac{1}{8}\sum_{s=1}^n \big( |T^i_{ks}|^2 + |T^k_{is}|^2 + 2\mbox{Re} \{ T^i_{is}\overline{T^k_{ks}} \} \big), \ \ \ \ \  \forall \ 1\leq i,k\leq n. 
\end{split}    \right.  \label{eq:iikk}
\end{equation}

As an immediate corollary, we have the following
\begin{proposition} \label{prop0}
Let $(M^n,g)$ be a Hermitian manifold. If its Chern and Levi-Civita holomorphic sectional curvatures are equal: $H=H^r$, then $g$ is K\"ahler.
\end{proposition}

\begin{proof}
The assumption says that $R^r_{X\overline{X}X\overline{X}}= R_{X\overline{X}X\overline{X}}$ for any type $(1,0)$ tangent vector $X$. By the first line of (\ref{eq:iikk}), we have $T^X_{Xs}=0$ for any $1\leq s\leq n$ and any $X=\sum_iX_ie_i$. This means
$$ \sum_{i,j=1}^n X_i \overline{X}_j T^j_{is}=0,\ \ \ \ \forall \ 1\leq s\leq n, $$
where these $X_i$ are arbitrary. Therefore $T^j_{ik}=0$ for any $1\leq i,j,k\leq n$, hence $T=0$ and $g$ is K\"ahler. 
\end{proof}

Next, let us recall some basic properties for Lie-Hermitian manifolds. Let $G$ be a connected, simply-connected, even-dimensional Lie group, and ${\mathfrak g}$ its Lie algebra. Left-invariant complex structures and  compatible left-invariant metrics  on $G$ correspond to almost complex structures $J$ and compatible inner products $g=\langle \,,\,\rangle$  on ${\mathfrak g}$, such that $J$ is integrable, namely,
\begin{equation*}
[x,y]- [Jx,Jy] +J [Jx, y] + J[x,Jy] =0, \ \ \ \ \ \forall \ x,y\in {\mathfrak g}.
\end{equation*}
Extend $J$ and $\langle \,,\,\rangle$ linearly over ${\mathbb C}$ to the complexification ${\mathfrak g}^{\mathbb C}$. We have decomposition ${\mathfrak g}^{\mathbb C} = {\mathfrak g}^{1,0} \oplus {\mathfrak g}^{0,1}$ into the $(1,0)$ and $(0,1)$ parts, where ${\mathfrak g}^{1,0}= \{ x-\sqrt{-1}Jx \mid x\in {\mathfrak g}\}$ and ${\mathfrak g}^{0,1} = \overline{{\mathfrak g}^{1,0}}$.
Write $\dim_{\mathbb R}{\mathfrak g} =2n$. A unitary basis $e=\{ e_1, \ldots , e_n\}$ of ${\mathfrak g}^{1,0}$ will be called a {\em unitary frame} of ${\mathfrak g}$ from now on. Following the notations of \cite{VYZ} (see also \cite{YZ-Gflat} and \cite{ZZJGP}), let us denote by
\begin{equation*}
C_{ik}^j = \langle [e_i,e_k] , \,\overline{e}_j \rangle , \ \ \ \ \ D_{ik}^j = \langle  [\overline{e}_j, e_k] , e_i \rangle, \ \ \ \ \ 1\leq i,j,k\leq n,
\end{equation*}
for the structure constants. By the integrability of $J$ we have
\begin{equation}
[e_i,e_j] = \sum_{k=1}^n C_{ij}^k e_k , \ \ \ \ \ [e_i,\overline{e}_j ] = \sum_{k=1}^n \big( \overline{D^i_{kj}} \,e_k -  D_{ki}^j \overline{e}_k \big).   \label{E2CD}
\end{equation}
We can extend each $e_i$  to left-invariant vector fields on $G$, which will still be denoted as $e_i$. So $e$ becomes a global unitary frame on $G$ as a Hermitian manifold. This will be our frame of choice from now on. Under this frame, the Chern connection $\nabla$ has expression
\begin{equation*}
\nabla e_i  = \sum_{j=1}^n \theta_{ij}e_j = \sum_{j=1}^n  \sum_{k=1}^n \big( D^j_{ik}\varphi_k  -\overline{D^i_{jk}} \overline{\varphi_k} \big) \,e_j \label{theta}
\end{equation*}
where $\varphi$ is the coframe of $(1,0)$-forms dual to $e$. As is well-known, the Chern torsion $T$ satisfies $T(X,\overline{Y})=0$ for any type $(1,0)$ tangent vectors $X$ and $Y$. Write the components of $T$ under the frame $e$ as $T^j_{ik}$, namely, $\, T(e_i,e_k)   = \sum_{j=1}^n  \,T_{ik}^j e_j$. Note that our $T^j_{ik}$ here is twice of the same notation in \cite{YZ} or \cite{VYZ}. We have (see \cite{VYZ} for instance, and notice the difference caused by the factor $2$)
\begin{equation}
T_{ik}^j = - C_{ik}^j - D_{ik}^j + D_{ki}^j.   \label{T}
\end{equation}
The covariant derivative of $T$ with respect to the Chern connection $\nabla$ are given by
\begin{equation*}
\left\{   \begin{split} T^j_{ik,\,\ell} = \sum_{s=1}^n \big( -T^j_{sk} D^s_{i\ell} - T^j_{is} D^s_{k\ell} + T^s_{ik} D^j_{s\ell } \big),   \\
T^j_{ik,\,\overline{\ell}} = \sum_{s=1}^n \big( T^j_{sk} \overline{D^i_{s\ell}} + T^j_{is} \overline{D^k_{s\ell} } - T^s_{ik} \overline{ D^s_{j\ell } }\big), \ \ \ \end{split}  \right. \label{Tderi}
\end{equation*}
for any $i$, $j$, $k$, $\ell$, where the indices after comma stand for covariant derivatives with respect to the Chern connection $\nabla$. Now the structure equation takes form:
\begin{equation}
d\varphi_{i} = -\frac{1}{2} \sum_{j,k}  C^i_{jk}\varphi_j \wedge \varphi_k - \sum_{j,k} \overline{D^j_{ik}} \varphi_{j} \wedge \overline{\varphi}_{k}, \ \ \ \ \ \ \forall \ 1\leq i\leq n.  \label{SED}
\end{equation}
By \cite[Lemma 2.1]{VYZ}, the structure constants $C$ and $D$ satisfy the first Bianchi identity, which in this case is also equivalent to the Jacobi identity:
\begin{equation}
\left\{   \begin{split}
 \ \sum_{s=1}^n \big( C^s_{ij}C^{\ell}_{sk} + C^s_{jk}C^{\ell}_{si} + C^s_{ki}C^{\ell}_{sj} \big) \ = \ 0, \hspace{3.2cm} \\
 \ \sum_{s=1}^n \big( C^s_{ik}D^{\ell}_{js} + D^s_{ji}D^{\ell}_{sk} - D^s_{jk}D^{\ell}_{si} \big) \ = \ 0 , \hspace{3cm}   \\
 \ \sum_{s=1}^n \big( C^s_{ik}\overline{D^{s}_{j\ell }} - C^j_{sk}\overline{D^{i}_{s\ell }} + C^j_{si}\overline{D^{k}_{s\ell }}  - D^{\ell}_{si}\overline{D^{k}_{js }} + D^{\ell}_{sk} \overline{ D^{i}_{js }}  \big) \ = \ 0,  
\end{split}  \right. \label{CCCD}
\end{equation}
for any $1\leq i,j,k,\ell \leq n$. For Lie-Hermitian manifolds, the components of the curvature tensor $R$ of $\nabla$ take a particularly simple form:

%\vspace{0.3cm}

\begin{lemma} [\cite{LZ}] \label{lemma3}
Under any unitary basis $e$ of ${\mathfrak g}$, the components of Chern curvature tensor $R$ of a Lie-Hermitian manifold are given by
\begin{equation*}
R_{i\overline{j}k\overline{\ell}} = \sum_{s=1}^n \big( D^s_{ki} \overline{ D^s_{\ell j} } - D^{\ell}_{si} \overline{ D^k_{sj } }  - D^j_{si} \overline{ D^k_{\ell s} } -  \overline{ D^i_{sj} } D^{\ell}_{ks}  \big) \label{R},  \ \ \ \ \ \ \forall \ 1\leq i,j,k,\ell \leq n.
\end{equation*}
\end{lemma}

In particular, we have
\begin{equation}
R_{i\overline{i}i\overline{i}} = \sum_{s=1}^n \big( |D^s_{ii}|^2 - |D^i_{si}|^2 - 2 \mbox{Re} \{  D^i_{si} \overline{D^i_{is}} \} \big) , \ \ \ \ \ \ \ \forall \ 1\leq i\leq n. \label{HD}
\end{equation}
By (\ref{Rhat}) and (\ref{R}), we immediately obtain the expression for $\widehat{R}$, and in particular, we have

\begin{lemma} [\cite{LZ}]  \label{lemma4}
Under any unitary basis $e$ of ${\mathfrak g}$, it holds that
\begin{equation*}
\widehat{R}_{i\overline{i}k\overline{k}} = \sum_{s=1}^n \big( |D^s_{ki}+D^s_{ik}|^2 - |D^k_{si}|^2 -|D^i_{sk}|^2 - 2 \mbox{Re} \{ D^k_{sk} \overline{D^i_{si} } +  D^i_{si} \overline{D^k_{ks} } + D^k_{sk} \overline{D^i_{is} } + D^i_{sk} \overline{D^i_{ks} } + D^k_{si} \overline{D^k_{is} }\} \big) , \label{RhatD}
\end{equation*}
for any $1\leq i,k \leq n$.
\end{lemma}

\vspace{0.3cm}

\section{The almost abelian case }

Let $(\mathfrak{g},J,g)$ be a Lie algebra equipped with a Hermitian structure. Throughout this section, we will assume that $\mathfrak{g}$ is {\em almost abelian,}  namely, $\mathfrak{g}$ (is not abelian but it) contains an abelian ideal $\mathfrak{a}$ of codimension one. It is equivalent to the condition that it contains an abelian subalgebra of codimension 1 (\cite{BC}). Our goal is to confirm Conjecture \ref{conj2} for ${\mathfrak g}$. We will follow the notations of \cite{GZ}.

Assume that $(\mathfrak{g},J,g)$ is an almost  abelian Lie algebra equipped with a Hermitian structure. Let $\mathfrak{a} \subset \mathfrak{g}$ be an abelian ideal of of codimension $1$. Then the $J$-invariant ideal $\mathfrak{a}_J :=\mathfrak{a} \cap J\mathfrak{a}$ is of codimension $2$ in $\mathfrak{g}$, so we can always choose a unitary basis $\{e_1,e_2,\ldots ,e_n\}$ of $\mathfrak{g}^{1,0}$ such that $\mathfrak{a}$ is spanned by
$$
\mathfrak{a}=\mbox{span}_\mathbb{R} \{\sqrt{-1}(e_1-\bar{e}_1);\,(e_i+\bar{e}_i), \,\sqrt{-1}(e_i-\bar{e}_i);\ 2 \leq i \leq n \}.
$$
We will call such a unitary basis an {\em admissible frame}. Following \cite{GZ}, since $\mathfrak{a}$ is abelian, we have
$$
[e_i,e_j]=[e_i,\bar{e}_j]=[e_1-\bar{e}_1,e_i]=0, \ \ \ \ \ \ \   \forall  \ 2 \leq i,j \leq n .
$$
From this and (\ref{E2CD}) we get that
$$
C^\ast_{ij}=D^j_{\ast i}=D^1_{\ast i}=C^\ast _{1i}+\overline{D^i_{\ast 1}}=0, \ \ \ \ \ \ \   \forall  \ 2 \leq i,j \leq n .
$$
Since $\mathfrak{a}$ is an ideal, $[e_1 + \bar{e}_1,\mathfrak{a}] \subseteq \mathfrak{a}$, we get $C^1_{1i}=0$ and $D^1_{11}=\overline{D^1_{11}}$. Putting these together, we know that the only possibly non-trivial components of $C$ and $D$ are
\begin{equation}
D^1_{11}=\lambda \in \mathbb{R}, \ \ \  D^1_{i1}=v_i \in \mathbb{C}, \ \ \  D^j_{i1}=A_{ij}, \ \ \  C^j_{1i}=-\overline {A_{ji}}, \ \ \ \ \ \  2 \leq i,j \leq n , \label{CDZ}
\end{equation}
where $\lambda \in \mathbb{R}$, $v \in \mathbb{C}^{n-1}$ is a column vector, and $A=(A_{ij})_{2\leq i,j \leq n}$ is an $(n-1)\times (n-1)$ complex matrix. Expressed in terms of the coframe $\varphi$ dual to $e$, the structure equation (\ref{SED}) now becomes
$$
\begin{cases}
    d \varphi_{1}=- \lambda \varphi_{1} \wedge \overline{\varphi}_{1}, \\
    d \varphi_{i}=- \overline{v}_{i} \varphi_{1} \wedge \overline{\varphi}_{1} -\sum^{n}_{j=2}(\varphi_{1}+\overline{\varphi}_{1})\wedge \varphi_{j}, \,\,\,\,\, 2 \leq i \leq n.
\end{cases}
$$
The following characterization result from \cite[Proposition 7 and Lemma 6]{GZ} for almost abelian Hermitian Lie algebras  will be needed in our proof of Theorem \ref{thm}:
\begin{lemma} [\cite{GZ}] \label{lemma5}
Let $\mathfrak{g}$ be an almost abelian Lie algebra equipped with a Hermitian structure $(J,g)$, and let $e$ be an admissible frame of ${\mathfrak g}$. Then the following hold:
\begin{enumerate}
\item ${\mathfrak g}$ is unimodular \ $\Longleftrightarrow$ \ $\lambda + \mbox{tr}(A) + \overline{\mbox{tr}(A)} =0$; 
\item $g$ is K\"ahler \ $\Longleftrightarrow$ \ $v=0$ and $A+A^{\ast}=0$;
\item $g$ is Chern flat \ $\Longleftrightarrow$ \ $\lambda =0$, $v=0$, and $[A,A^{\ast}]=0$.
\end{enumerate}
Here $A^{\ast}$ stands for the conjugate transpose of the matrix $A$. 
\end{lemma}

From \cite{GZ}, we also know that under any admissible frame $e$, the only possibly non-trivial Chern torsion and curvature components are
\begin{equation}
\left\{ \begin{split}
T^1_{1i}=v_i,\ \ \ \ \ T^j_{1i}= A_{ij}+\overline{A_{ji}} , \hspace{1.9cm} \\
R_{1\bar{1}1\bar{1}} = -2\lambda^2-|v|^2, \ \ \ \ \ (R_{1\bar{1}i\bar{1}}) = - A^{\ast}v,  \\
(R_{1\bar{1}i\bar{j}}) = v v^{\ast} + [A, A^{\ast}] -\lambda (A+A^{\ast}),  \hspace{0.7cm}
\end{split}
\right.  \label{eq:aaTR}
\end{equation}
for any $2\leq i,j\leq n$. From this, we deduce the following:

\begin{proposition} \label{prop2}
Let $({\mathfrak g}, J,g)$ be an almost abelian Lie algebra equipped with a Hermitian structure. If the Chern holomorphic sectional curvature $H$ is a constant $c$,  then $c=0$ and $g$ is Chern flat.
\end{proposition}

\begin{proof}
By (\ref{eq:aaTR}), we have $R_{1\bar{1}1\bar{1}} = -2\lambda^2-|v|^2$ and $R_{i\bar{i}i\bar{i}}=0$ are any $2\leq i\leq n$. So when $H=c$, we get $c=0$, and $\lambda =0$, $v=0$. On the other hand, by (\ref{eq:aaTR}) again we have
$$ 4\widehat{R}_{1\bar{1}i\bar{j}} =  R_{1\bar{1}i\bar{j}} = ([A, A^{\ast}])_{ij}, \ \ \ 2\leq i,j\leq n. $$ 
Therefore  by $c=0$ and (\ref{RhatC}) we get $[A, A^{\ast}]=0$, hence $g$ is Chern flat by part (iii) of Lemma \ref{lemma5}. This completes the proof of the proposition. Note that here we did not need to assume that ${\mathfrak g}$ is unimodular.
\end{proof}

Next let us consider the constant holomorphic sectional curvature condition for Levi-Civita connection. We have the following:
\begin{proposition} \label{prop3}
Let $({\mathfrak g}, J,g)$ be an almost abelian Lie algebra equipped with a Hermitian structure. Assume that the Levi-Civita holomorphic sectional curvature $H^r$ is a constant $c$  and ${\mathfrak g}$ is unimodular. Then $c=0$ and $g$ is K\"ahler flat.
\end{proposition}

\begin{proof}
Write $S=A+A^{\ast}$. By a unitary change of $\{ e_2, \ldots , e_n\}$ if necessary, we may assume that $S$ is a real diagonal matrix. Since $H^r=c$, by (\ref{RrhatC}) we have
$$ R^r_{ 1\bar{1}1\bar{1} } = R^r_{ i\bar{i}i\bar{i} } =c, \ \ \ \  \widehat{R}^r_{ 1\bar{1}i\bar{i} } =  \widehat{R}^r_{ i\bar{i}k\bar{k} }= \frac{c}{2}, \ \ \ \ \ \ \forall \ 2 \leq i\neq k\leq n. $$
Therefore for any $2\leq i\neq k\leq n$, by (\ref{eq:iikk}) and the fact that $S$ is real diagonal we get
\begin{eqnarray*}
  c & = & R^r_{1\bar{1}1\bar{1}} \ = \   R_{1\bar{1}1\bar{1}} - \frac{1}{2}\sum_s |T^1_{1s}|^2 \ = \ -2\lambda^2-\frac{3}{2}|v|^2, \\
 c & = &  R^r_{i\bar{i}i\bar{i}} \ = \   R_{i\bar{i}i\bar{i}} - \frac{1}{2}\sum_s |T^i_{is}|^2 \ = \  - \frac{1}{2} |S_{ii}|^2 , \\
\frac{c}{2} & = & \widehat{R}^r_{i\bar{i}k\bar{k}} \ = \  \widehat{R}_{i\bar{i}k\bar{k}} - \frac{1}{8}\sum_s \big( |T^k_{is}|^2 + |T^i_{ks}|^2 + 2\mbox{Re} \{ T^i_{is} \overline{ T^k_{ks}} \} \big) \ = \ -\frac{1}{4} S_{ii}S_{kk}, \\
 \frac{c}{2} & = & \widehat{R}^r_{1\bar{1}i\bar{i}} \ = \  \widehat{R}_{1\bar{1}i\bar{i}} - \frac{1}{8}\sum_s \big( |T^1_{is}|^2 + |T^i_{1s}|^2 + 2\mbox{Re} \{ T^1_{1s} \overline{ T^i_{is}} \} \big) \ = \ \\
 & = & \frac{1}{4}\big( |v_i|^2 + [A, A^{\ast}]_{ii} - \lambda S_{ii}\big)  - \frac{1}{8}( |v_i|^2 + |S_{ii}|^2).
 \end{eqnarray*}
The second and third lines imply that $c\leq 0$ and $S_{ii}=\delta \sqrt{-2c}$ for each $2\leq i\leq n$, where $\delta \in \{ 1, -1\}$. By summing $i$ from $2$ to $n$ in the last line, we get
$$ \frac{c}{2}(n-1) = \frac{1}{8} |v|^2 - \lambda \mbox{tr}(S) + \frac{c}{4}(n-1), $$
or equivalently,
$$ 0\geq \frac{c}{4}(n-1) = \frac{1}{8} |v|^2 - \lambda \mbox{tr}(S). $$
When ${\mathfrak g}$ is assumed to be unimodular, we have $\mbox{tr}(S)=-\lambda$ by part (i) of Lemma \ref{lemma5}, so the far right hand side of the above becomes $ \frac{1}{8} |v|^2 + \lambda^2$ which is non-negative, hence we conclude that $c=0$, $v=0$, $\lambda =0$, and $S=0$. This means that $g$ is K\"ahler and flat, and we have completed the proof of Proposition \ref{prop3}. 
\end{proof}

\noindent {\bf Remark:} (1). Note that the combination of Propositions \ref{prop2} and \ref{prop3} immediately lead to the proof of Theorem \ref{thm} in the almost abelian case. (2). If we assume $c=0$ in Proposition \ref{prop3}, then the argument in the proof implies that $v=0$, $\lambda =0$, $S=0$. Therefore, {\em any almost abelian Hermitian Lie algebra with vanishing Levi-Civita holomorphic sectional curvature must be K\"ahler flat (and unimodular).} (3). From the proof of Proposition \ref{prop3} we also see that, without the assumption that ${\mathfrak g}$ is unimodular, then it is possible that $c$ might take negative values. For instance, when $c=-2$, $\lambda =1$, $v=0$, and each $S_{ii}=2$. 

\vspace{0.2cm}

In the following, we will give an explicit example of a non-unimodular almost abelian Hermitian Lie algebra, whose Levi-Civita holomorphic sectional curvature is a negative constant, yet it is non-K\"ahler (hence it is not a complex space form).  This again illustrates the point we mentioned earlier that the compactness assumption in Conjectures \ref{conj1} or \ref{conj2} is  a necessary one. 

\vspace{0.2cm}

\noindent {\bf Example:} Fix any integer $n\geq 2$. Let ${\mathfrak g}= {\mathbb R}\{ X, Y, Z_3, \ldots , Z_{2n}\} $ be the  Lie algebra with Lie brackets 
$$ [X,Y]=\sqrt{2} Y, \ \ \ [X,Z_3]=-\sqrt{2}Z_3, \ \ldots , \ [X,Z_{2n}]=-\sqrt{2}Z_{2n}, $$
while all other Lie brackets are zero. Clearly ${\mathfrak a}= {\mathbb R}\{  Y, Z_3, \ldots , Z_{2n}\} $ is an abelian ideal, so  ${\mathfrak g}$ is almost abelian. Let $g=\langle ,\rangle$ be the metric so the above basis is orthonormal, and $J$ the almost complex structure defined by 
$$ JX=Y, \ \ \ JZ_{2i-1}=Z_{2i}, \ \ \ 2\leq i\leq n. $$
Then clearly $J$ is integrable, so $(J,g)$ becomes a Hermitian structure on ${\mathfrak g}$. Note that the trace of $\mbox{ad}_X$ is equal to $-(2n-3)\sqrt{2}$ which is non-zero, hence  ${\mathfrak g}$ is not unimodular. Let us verify that it has constant negative Levi-Civita holomorphic sectional curvature. Write
$$ e_1=\frac{1}{\sqrt{2}}(X-\sqrt{-1}Y), \ \ \ e_i=\frac{1}{\sqrt{2}}(Z_{2i-1}-\sqrt{-1}Z_{2i}),\ \ \ 2\leq i\leq n, $$
then $e$ becomes an admissible frame. We compute that
$$   [e_1, \overline{e}_1] = - e_1 + \overline{e}_1, \ \ \   [e_1, e_i]=-e_i, \ \ \ [e_1, \overline{e}_i] = - \overline{e}_i, \ \ \ \ 2\leq i\leq n. $$
Therefore the structure constants are $\lambda =1$, $v=0$, and $A=I$. It is a straight-forward computation to show that $H^r=-2$ is a negative constant, while $T^i_{1i}=2$ for each $2\leq i\leq n$, so the metric is not K\"ahler.

\vspace{0.3cm}

\section{Lie algebra with $J$-invariant abelian ideal of codimension 2}
In this section, we will restrict ourselves to Lie algebras which contain $J$-invariant abelian ideals of codimension $2$. Throughout this section, we will let $\mathfrak{g}$ be a Lie algebra equipped with a Hermitian structure $(J,g)$ such that $\mathfrak{a}\subset {\mathfrak g}$ is an abelian ideal  of codimension 2 which is  $J$-invariant.  Note that such a $\mathfrak{g}$ is always almost abelian, and it is always $3$-step solvable, but in general it might not be $2$-step solvable. 

Following \cite{GZ}, we will call a unitary basis $\{e_{1},e_{2}, \ldots , e_{n}\}$ of $\mathfrak{g}^{1,0}$ an {\em admissible frame} of ${\mathfrak g}$, if 
$$
\mathfrak{a}= \mbox{span}_{\mathbb{R}} \{e_{i}+\bar{e}_{i}, \sqrt{-1}(e_{i}-\bar{e}_{i});\ 2 \leq i \leq n\},
$$
and $D^1_{11}\geq 0$. The latter can be achieved by rotating $e_1$ by a complex number of norm $1$. Since $\mathfrak{a}$ is an abelian ideal, we have
$$
C^{\ast}_{ij}=D^{j}_{\ast i}=C^{1}_{\ast \ast}=D^{i}_{1 \ast}=D^{\ast}_{1 i}=0,\,\,\,\,\,\,\, \forall \ 2 \leq i,j \leq n,
$$
hence the only possibly non-zero components of $C$ and $D$ are
\begin{equation} \label{XYZv}
C^{j}_{1i}=X_{ij},\,\,\,\, D^{1}_{11}=\lambda,\,\,\,\, D^{j}_{i1}=Y_{ij},\,\,\,\, D^{1}_{ij}=Z_{ij},\,\,\,\, D^{1}_{i1}=v_{i},\,\,\,\,\,\,\forall \ 2 \leq i ,j \leq n,
\end{equation}
where $\lambda \geq 0$, $v \in \mathbb{C}^{n-1}$ is an column vector, and $X$, $Y$, $Z$ are $(n-1)\times (n-1)$ complex matrices. Denote by $\{ \varphi_1, \ldots , \varphi_n\} $ the coframe of $(1,0)$-forms dual to $e$, and for convenience let us write $\varphi =\,^t\!(\varphi_2, \ldots , \varphi_n)$ as a column vector. Then the structure equation becomes
$$
\begin{cases}
    d \varphi_{1}=- \lambda \varphi_{1} \overline{\varphi}_{1} \\
    d \varphi=- \varphi_{1} \overline{\varphi}_{1} \bar{v}-\varphi_{1}\,  ^t\!X \varphi+\overline{\varphi}_{1} \overline{Y} \varphi -\varphi_{1} \overline{Z} \overline{\varphi}
\end{cases}
$$
By the equations (\ref{CCCD}), we know that these structure constants obey the following
\begin{equation}  \label{XYZ}
 \begin{cases}
    \lambda (X^{\ast}+Y)+[X^{\ast},Y]-Z \overline{Z}=0 \\
    \lambda Z-(Z \,^{t}\!X+YZ)=0                
 \end{cases}
\end{equation}
By (\ref{T}), the only possibly non-zero Chern torsion components are
\begin{equation} \label{codim2-torsion}
T^{1}_{1 i}=v_{i},\,\,\,\,\, T^{1}_{ij}=Z_{ji}-Z_{ij},\,\,\,\,\, T^{j}_{1i}=Y_{ij}-X_{ij},\,\,\,\,\,\,\,\, 2 \leq i , j\leq n.
\end{equation}
The following properties were established in \cite{GZ}:
\begin{lemma} [\cite{GZ}] \label{lemma-codim2}
Let $(\mathfrak{g},J,g)$ be a Lie algebra with a Hermitian structure and let $\mathfrak{a}\subset {\mathfrak g}$ be an abelian ideal of codimension $2$ that is  $J$-invariant. Let $e$ be an admissible frame and the possibly non-trivial structure constants are given by (\ref{XYZv}). Then the following hold:
\begin{enumerate}
\item ${\mathfrak g}$ is unimodular \ $\Longleftrightarrow$ \ $\lambda + \mbox{tr}(Y\!-\!X)=0$;
\item $g$ is K\"ahler \ $\Longleftrightarrow$ \ $v=0$, $\,^t\!Z=Z$, and $X=Y$;
\item $g$ is Chern flat \ $\Longleftrightarrow$ \  $\lambda =0$, $v=0$, $Z=0$, $[Y,Y^{\ast}]=0$, and $[Y,X^{\ast}]=0$. 
\end{enumerate}
\end{lemma}

Next, let us compute the components of the Chern curvature of $g$ under an admissible frame. By (\ref{Rhat}), Lemma \ref{lemma3}, and (\ref{XYZv}), a straight-forward computation leads to the following:
\begin{equation} \label{ChernRhat1}
\left\{ \begin{split} R_{1\bar{1}1\bar{1}} \, = \, -2\lambda^2 -|v|^2, \ \ \  
R_{i\bar{i}i\bar{i}} \, = \, |Z_{ii}|^2, \ \ \ 
\widehat{R}_{i\bar{i}k\bar{k}}  \, = \, \frac{1}{4}|Z_{ik}+Z_{ki}|^2, \ \ \ \ \forall \ 2\leq i\neq k\leq n,\\
\widehat{R}_{1\bar{1}i\bar{j}}  \, = \, \frac{1}{4} \big( vv^{\ast} +[Y,Y^{\ast}] -\lambda (Y+Y^{\ast}) -Z\overline{Z} -\,^t\!ZZ^{\ast} - \,^t\!Z\overline{Z} \big)_{ij}\,,  \ \ \ \ \forall \ 2\leq i, j\leq n,  \hspace{0.5cm}\\
\sum_{i=2}^n \widehat{R}_{1\bar{1}i\bar{i}}  \, = \, \frac{1}{4} \big( |v|^2  -\lambda \,\mbox{tr}(Y+Y^{\ast}) - 2\,\mbox{tr}(Z\overline{Z}) -|Z|^2\big), \hspace{4.6cm}\\
\end{split}
\right. 
\end{equation}
where the last line is obtained by taking trace in the line above. Now we are ready to prove the Chern case of Theorem \ref{thm} for ${\mathfrak g}$. Here again we do not need to assume that ${\mathfrak g}$ is unimodular.

\begin{proposition}  \label{prop4}
Let $({\mathfrak g},J,g)$ be a Hermitian Lie algebra which contains a $J$-invariant  abelian ideal of codimension $2$. Assume that the Chern holomorphic sectional curvature of $g$ is a constant $c$. Then $c=0$ and $g$ is Chern flat.
\end{proposition} 

\begin{proof}
From the assumption, we have $R_{1\bar{1}1\bar{1}}=R_{i\bar{i}i\bar{i}} =c$ and $\widehat{R}_{i\bar{i}k\bar{k}}=\frac{c}{2}$ for any $2\leq i\neq k\leq n$. By the first line of (\ref{ChernRhat1}), we conclude that $c=0$ and $\lambda =0$, $v=0$, $Z_{ii}=0$, $Z_{ik}+Z_{ki}=0$ for each $2\leq i \neq k \leq n$. Thus $Z$ is skew-symmetric. By taking trace on the first line of (\ref{XYZ}), we get $\mbox{tr}(Z\overline{Z})=0$. Therefore 
$$ |Z|^2 = \mbox{tr}(\,^t\!Z\overline{Z}) = -\mbox{tr}(Z\overline{Z}) =0, $$
hence $Z=0$. The equation (\ref{XYZ}) now gives us $[X^{\ast}, Y]=0$. Finally, the second line in (\ref{ChernRhat1}) gives us $[Y,Y^{\ast}]=0$. So by part (iii) of Lemma \ref{lemma-codim2} we know that $g$ is Chern flat. This completes the proof of the proposition.
\end{proof}

We are now left with the Levi-Civita case for ${\mathfrak g}$, which is the main technical part of this article. Since almost abelian Lie algebras are special cases of our ${\mathfrak g}$, we know that ${\mathfrak g}$ must be assumed to be unimodular. From now on, let us assume that $({\mathfrak g},J,g)$ is a unimodular Hermitian Lie algebra which contains a $J$-invariant abelian ideal of codimension $2$, such that the Levi-Civita holomorphic sectional curvature $H^r$ of $g$ is a constant $c$. Our goal is to show that $c$ must be zero and $g$ must be Levi-Civita flat. It turns out that $g$ is actually  K\"ahler flat in this case. 

Let $e$ be an admissible frame. Write $B=Y-X$. Note that when we change the basis $\{ e_2, \ldots , e_n\}$ be a unitary matrix $U$, the matrix $Z$ is changed into $UZ\,^t\!U$. Since $\,^t\!Z+Z$ is a symmetric $(n-1)\times (n-1)$ matrix, as is well-known, there exists a unitary matrix $U$ such that $U(^t\!Z+Z)\,^t\!U$ is diagonal with non-negative entries. So by a unitary change of $\{e_2, \ldots , e_n\}$ is necessary, we may assume that $\,^t\!Z+Z$ is diagonal and with real non-negative diagonal entries. We will fix such an admissible frame from now on. By Lemma \ref{lemma2} and (\ref{codim2-torsion}) - (\ref{ChernRhat1}), we get
\begin{equation} \label{LeviRhat1}
\left\{ \begin{split} R^r_{1\bar{1}1\bar{1}} \, = \, -2\lambda^2 -\frac{3}{2}|v|^2, \ \ \  \ \ \ \ \ \ 
R^r_{i\bar{i}i\bar{i}} \, = \, |Z_{ii}|^2    - \frac{1}{2} |B_{ii}|^2, \ \ \  \ \ \ \forall \ 2\leq i\leq n, \hspace{1.75cm}\\
\widehat{R}^r_{i\bar{i}k\bar{k}}  \, = \,  - \frac{1}{8} \big(   |B_{ik}|^2 + |B_{ki}|^2 + 2 \mbox{Re} (B_{ii} \overline{B}_{kk} ) \big) ,  \ \ \ \ \ \  \forall \ 2\leq i\neq k\leq n, \hspace{2.45cm}  \\
\sum_{i=2}^n \widehat{R}^r_{1\bar{1}i\bar{i}}  \, = \, \frac{1}{4} \big( |v|^2  -\lambda \,\mbox{tr}(Y+Y^{\ast}) - 2\,\mbox{tr}(Z\overline{Z}) -|Z|^2\big) - \frac{1}{8} \big( |v|^2 +|B|^2 + |\,^t\!Z-Z|^2 \big) .\\
\end{split}
\right. 
\end{equation}
 Note that in the middle line we used the convention that $\,^t\!Z+Z$ is diagonal. 
Since ${\mathfrak g}$ is unimodular, by part (i) of Lemma \ref{lemma-codim2} we get $\mbox{tr}(B)=-\lambda$. Taking trace in the first line of (\ref{XYZ}) and using the fact $X=Y-B$, we obtain
\begin{equation} \label{traceL}
\mbox{tr}(Z\overline{Z}) = \lambda \,\mbox{tr}(X^{\ast}+Y) = \lambda \,\mbox{tr}(Y^{\ast}+Y) +\lambda^2.
\end{equation}

Now we are ready to state and prove the following proposition, which is the main technical part of the article:

\begin{proposition}  \label{prop5}
Let $({\mathfrak g},J,g)$ be a unimodular Hermitian Lie algebra which contains a $J$-invariant  abelian ideal of codimension $2$. Assume that the Levi-Civita holomorphic sectional curvature of $g$ is a constant $c$. Then $c=0$ and $g$ is K\"ahler flat.
\end{proposition} 

\begin{proof}
We start from the assumption that $H^r=c$. Fix an admissible frame $e$ so that $^t\!Z+Z$ is diagonal with non-negative entries. By (\ref{LeviRhat1}) we have
\begin{eqnarray} 
& & c \ = \ -2\lambda^2 -\frac{3}{2}|v|^2 \ = \ |Z_{ii}|^2    - \frac{1}{2} |B_{ii}|^2 ,  \ \ \ 2\leq i\leq n, \label{eq:L1}\\
&& \frac{c}{2}(1+\delta_{ik} ) \ = \ \frac{1}{4} |Z_{ik}+Z_{ki}|^2 - \frac{1}{8} \big(   |B_{ik}|^2 + |B_{ki}|^2 + 2 \mbox{Re} (B_{ii} \overline{B}_{kk} ) \big),  \ \ \ 2\leq i, k\leq n, \label{eq:L1b} \\
&& \frac{c}{2}(n-1) \ = \  \frac{1}{4} \big( |v|^2  -\lambda \,\mbox{tr}(Y+Y^{\ast}) - 2\,\mbox{tr}(Z\overline{Z}) -|Z|^2\big) - \frac{1}{8} \big( |v|^2 +|B|^2 + |^t\!Z-Z|^2 \big).   \label{eq:L2}
\end{eqnarray}
Plugging (\ref{traceL}) into (\ref{eq:L2}), we get
$$ -4c(n-1)+|v|^2+2\lambda^2 -|B|^2 \ = \ 6\,\mbox{tr}(Z\overline{Z}) + 2|Z|^2 + |^t\!Z-Z|^2. $$
Since $|^t\!Z-Z|^2=2|Z|^2 - 2\,\mbox{tr}(Z\overline{Z})$ and  $|^t\!Z+Z|^2=2|Z|^2 + 2\,\mbox{tr}(Z\overline{Z})$, the right hand side of the line above is equal to $2|^t\!Z+Z|^2= 8\sum_{i=2}^n |Z_{ii}|^2$, so the equality becomes
\begin{equation} \label{eq:L3}
4c(n-1) \, = \, |v|^2 + 2\lambda^2 - 2 |^t\!Z+Z|^2 - |B|^2.
\end{equation}
The first equality of (\ref{eq:L1}) implies that $c\leq 0$. We proceed with the proof by establishing the following three claims:

\vspace{0.1cm}

\noindent {\bf Claim 1:} If $c=0$, then $g$ is K\"ahler flat.

\vspace{0.1cm}

In this case, by the first equality of (\ref{eq:L1}), we get that  $\lambda =0$ and $v=0$. By (\ref{traceL}) we obtain $\mbox{tr}(Z\overline{Z})=0$. Now  by  (\ref{eq:L2}) we conclude that $Z=0$ and $B=0$. The first equation of (\ref{XYZ}) tells us $[Y,Y^{\ast}]=0$, so $g$ is K\"ahler and flat by part (ii) and (iii) of Lemma \ref{lemma-codim2}. 

\vspace{0.1cm}

\noindent {\bf Claim 2:} If $n\geq 3$, then $c=0$.

\vspace{0.1cm}

Summing up $i$ and $k$ from $2$ to $n$ in (\ref{eq:L1b}) and using the fact that $\mbox{tr}(B)=-\lambda$,  we get
\begin{equation} \label{eq:L4}
2c\,n(n-1) \, = \, |^t\!Z+Z|^2 - |B|^2 - \lambda^2.
\end{equation}
Subtract (\ref{eq:L4}) from (\ref{eq:L3}), we obtain
$$ 2c(n-1)(2-n)=|v|^2+ 3\lambda^2 - 3|^t\!Z+Z|^2 \leq \frac{9}{4}|v|^2+ 3\lambda^2 =\frac{3}{2}(\frac{3}{2}|v|^2+2\lambda^2) = -\frac{3}{2}c.$$
If $c<0$, then the above inequality gives us
 $2(n-1)(2-n) \geq - \frac{3}{2}$, or equivalently, $(n-1)(n-2)\leq \frac{3}{4}<1$, contradicting with the assumption that $n\geq 3$, so the claim is proved.
 
 \vspace{0.1cm}

\noindent {\bf Claim 3:} If $n=2$, then $c=0$.

\vspace{0.1cm}
 
In this case, $B$ and $Z$ are $1\times 1$ matrices, and we have $B_{22}=\mbox{tr}(B)=-\lambda$, thus $|B|^2=\lambda^2$. Now the equations (\ref{eq:L1})  gives us
 $$ c=-2\lambda^2 -\frac{3}{2}|v|^2 = |Z|^2    - \frac{1}{2}\lambda^2 .$$
 By eliminating  and $\lambda^2$, we end up with
 $ 3c = 4|Z|^2 + \frac{3}{2}|v|^2$, which implies that $c\geq 0$ hence $c=0$. 
 
Putting these claims together, we have completed the proof of Proposition \ref{prop5}.
\end{proof}

The combination of Propositions \ref{prop4} and \ref{prop5} give a proof to Theorem \ref{thm} in the Levi-Civita case. 

\vspace{0.5cm}

\noindent\textbf{Acknowledgments.} {We would like to thank Haojie Chen, Shuwen Chen, Lei Ni, Xiaolan Nie, Kai Tang, Bo Yang, Yashan Zhang, and Quanting Zhao for their interests and/or helpful discussions.}

%%%%%%%%%%%%%%%%%%%%%%%%%%%%%%%%%%%%%%%%%%%%%%%%%%%%%%%
%%% Appendix sections. 附录章节, 非必选
%%%%%%%%%%%%%%%%%%%%%%%%%%%%%%%%%%%%%%%%%%%%%%%%%%%%%%%


\vspace{0.5cm}



\begin{thebibliography}{99}











\bibitem{ABD} A. Andrada, M. L. Barberis and I. Dotti, \emph{Abelian Hermitian geometry,} Differential Geom. Appl. {\bf 30} (2012), no.\,5, 509-519.


\bibitem{AOUV} D. Angella, A. Otal, L. Ugarte, R. Villacampa,  \emph{On Gauduchon connections with K\"ahler-like curvature,}  Commun. Anal. Geom. {\bf 30} (2022), no.\,5, 961-1006.




\bibitem {ADM} V. Apostolov, J. Davidov, and O. Muskarov, \emph{Compact self-dual Hermitian surfaces,}  Trans. Amer. Math. Soc. {\bf 348} (1996), 3051-3063.

\bibitem {Balas} A. Balas, \emph{Compact Hermitian manifolds of constant holomorphic sectional curvature,} Math. Zeit. {\bf 189} (1985), 193-210.

\bibitem {BG} A. Balas and P. Gauduchon, \emph{Any Hermitian metric of constant nonpositive (Hermitian) holomorphic sectional curvature on a  compact complex surface is K\"ahler,} Math. Zeit. {\bf 190} (1985), 39-43.


\bibitem{BDF} M. L. Barberis, I. Dotti, and A. Fino, \emph{Hyper-K\"ahler quotients of solvable Lie groups.} J. Geom. Phys. {\bf 56} (2006), no.\,4, 691-711.



\bibitem {Boothby} W. Boothby, \emph{Hermitian manifolds with zero
curvature,} Michigan Math. J. {\bf 5} (1958), no.\,2, 229-233.

\bibitem {BC} D. Burde and M. Ceballos, \emph{Abelian ideals of maximal dimension for solvable lie algebras,} J. Lie
Theory {\bf 22} (2012), no.\,3, 741-756. 



\bibitem {CCN} H. Chen, L. Chen, and X. Nie, \emph{Chern-Ricci curvatures, holomorphic sectional curvature and  Hermitian metrics,}  Sci. China Math. {\bf 64} (2021), 763-780.

\bibitem {CZ1} S. Chen and F. Zheng, \emph{ On Strominger space forms,} J Geom. Anal. {\bf 32} (2022), no.\,4, Paper No.141, 21pp 

\bibitem {CZ2} S. Chen and F. Zheng, \emph {Bismut torsion parallel metrics with constant holomorphic sectional curvature,} arxiv:2405.09110.

\bibitem {CZ3} S. Chen and F. Zheng, \emph {Canonical metric connections with constant holomorphic sectional curvature,}  arxiv:2501.03032.


\bibitem{CFGU} L. Cordero, M. Fern\'{a}ndez, A. Gray, L. Ugarte, \emph{Compact nilmanifolds with nilpotent complex structures: Dolbeault cohomology,} Trans. Amer. Math. Soc. {\bf 352} (2000), no.\,12, 5405-5433.



\bibitem{DGM} J. Davidov, G. Grantcharov, and O. Muskarov, \emph{Curvature properties of the Chern connection of twistor spaces,} 
Rocky Mt. J. Math. {\bf 39} (2009), no.\,1,  27-48. 






\bibitem {GZ} Y. Guo and F. Zheng, \emph{Hermitian geometry of Lie algebras with abelian ideals of codimension 2,} Math. Zeit. {\bf 304} (2023), no.\,3, Paper No 51, 24pp. 
    
%\bibitem {Gauduchon} P. Gauduchon, \emph{La {$1$}-forme de torsion d'une vari\'et\'e hermitienne compacte.} Math. Ann. \textbf{267} (1984), no. 4, 495-518.


\bibitem {Gauduchon1} P. Gauduchon, \emph{Hermitian connections
and Dirac operators,} Boll. Un. Mat. Ital. B (7) {\bf 11} (1997), no.\,2, suppl., 257-288.

\bibitem{Hano} J. Hano, \emph{On K\"ahlerian homogeneous spaces of unimodular Lie groups,} Amer. J. Math. {\bf 79} (1957), 885-900.


\bibitem{HZ} X. Huang and F. Zheng, \emph{On solvmanifolds with complex commutator and constant holomorphic sectional curvature,} arXiv: 2501.00810, to appear in Intern. J. Math.


\bibitem {KYZ} G. Khan, B. Yang, and F. Zheng, \emph{The set of all orthogonal complex strutures on the flat $6$-torus,} Adv. Math. {\bf 319} (2017), 451-471.

\bibitem {LZ} Y. Li and F. Zheng, \emph{Complex nilmanifolds with constant holomorphic sectional curvature,} Proc. Amer. Math. Soc. {\bf 150} (2022), 319-326. 
    

\bibitem {Milnor} J. Milnor, \emph {Curvatures of left invariant metrics on Lie groups,}  Adv. Math. {\bf 21} (1976), no.\,3, 293-329.

\bibitem {RZ} P. Rao and F. Zheng, \emph{Pluriclosed manifolds with constant holomorphic sectional curvature, } Acta. Math. Sinica (English Series) {\bf 38} (2022), no.\,6, 1094-1104. 

\bibitem {Salamon} S. M. Salamon, \emph{Complex structures on nilpotent Lie algebras,}
J. Pure Appl. Algebra {\bf 157} (2001), no.\,2-3, 311-333.

\bibitem {SS} T. Sato, K. Sekigawa, \emph{Hermitian surfaces of constant holomorphic sectional curvature,} Math. Zeit. {\bf 205} (1990), 659-668.

\bibitem {Tang} K. Tang, \emph{Holomorphic sectional curvature and K\"ahler-like metric,} Sci. China Math. (Chinese series) {\bf 50} (2020), 1-12.


\bibitem {VYZ} L. Vezzoni, B. Yang and F. Zheng, {\em Lie groups with flat Gauduchon connections,} Math. Zeit. {\bf 293} (2019), 597-608.

\bibitem {YZ} B. Yang and F. Zheng, {\em On curvature tensors of Hermitian manifolds,}  Comm. Anal. Geom. {\bf 26} (2018), no.\,5, 1195-1222.

\bibitem {YZ-Gflat} B. Yang and F. Zheng, {\em On compact Hermitian manifolds with a flat Gauduchon connection,}  Acta. Math Sinica (English Series) {\bf 34} (2018), no.\,8, 1259-1268.




\bibitem  {YangZ} X. Yang and F. Zheng, {\em On real bisectional curvature for Hermitian manifolds,} Trans. Amer. Math. Soc. {\bf 371} (2019),  2703-2718.

\bibitem{YZZ} S.-T. Yau, Q. Zhao, F. Zheng,  \emph{On Strominger K\"ahler-like manifolds with degenerate torsion,} Trans. Amer. Math. Soc. {\bf 376} (2023), no.\,5, 3063-3085.



\bibitem {ZZJGP}  Q. Zhao and F. Zheng, \emph {Complex nilmanifolds and Kahler-like connections,} J. Geom. Phys. {\bf 146} (2019), 103512, 9pp.

\bibitem{ZZCrelle} Q. Zhao, F. Zheng, \textit{Strominger connection and pluriclosed metrics}, J. Reine Angew. Math. {\bf 796} (2023), 245-267.


\bibitem{ZhaoZ24} Q. Zhao, F. Zheng, \emph{Curvature characterization of Hermitian manifolds with Bismut parallel torsion,} arXiv: 2407.10497.

\bibitem{ZhouZ} W. Zhou, F. Zheng, \emph{Hermitian threefolds with vanishing real bisectional curvature,} Sci. China Math. (Chinese series) {\bf 52} (2022), 757-764.



\end{thebibliography}
\end{document}